\documentclass[10pt,leqno]{amsart}
\usepackage{graphicx}
\usepackage{indentfirst,csquotes}

\usepackage{amssymb,amsthm,amsmath}
\usepackage{xcolor,paralist,hyperref,titlesec,fancyhdr,etoolbox}
\usepackage{bbm}
\usepackage{enumitem}
\usepackage{mathtools}
\usepackage[mathscr]{eucal}
\usepackage[english]{babel}
\newtheorem{theorem}{Theorem}[]

\newtheorem{lemma}[theorem]{Lemma}

\theoremstyle{remark}
\newtheorem*{remark}{Remark}

\titleformat{\section}[display]{\normalfont\huge\bfseries\centering}{\centering\chaptertitlename\thechapter}{10pt}{\Large}
\titlespacing*{\section}{0pt}{0ex}{0ex}

\hypersetup{ colorlinks=true, linkcolor=black, filecolor=black, urlcolor=black }

\usepackage{lipsum}

\begin{document}
\title{Focusing solutions of Vlasov-Poisson equation on a spherical shell} 
\author[Initial Surname]{Kaiwen Tan}
\date{\today}
\email{kevin.tan0606@outlook.com}
\maketitle

\let\thefootnote\relax

\begin{abstract}
We study smooth, spherically-symmetric solutions to the Vlasov-Poisson system and relativistic Vlasov-Poisson system in the plasma physical case. In particular, we construct solutions that initially possess arbitrarily small charge densities and electric fields, but attain arbitrarily large value of these quantities and concentrate on any given spherical shell at some later time. 
\end{abstract} 

\bigskip
\section*{1. Introduction}
We consider the single-species Vlasov-Poisson (VP) system:
\begin{equation} \label{VP-system}
    \begin{aligned}
        &\partial_t f + v \cdot \triangledown f + E \cdot \triangledown_v f = 0, \\ 
        &\rho(t, x) = \int_{\mathbb{R}^3} f(t, x, v) dv,\\
        &E(t, x) = \int_{\mathbb{R}^3} \frac{x-y}{| x-y |^3}\rho (t, y) dy.
    \end{aligned}
\end{equation}
Here $t>0$ represents time, $x \in \mathbb{R}^3$ is the particle position, $v \in \mathbb{R}^3$ is the particle momentum. Also, $f(t, x, v) >0$ is the density distribution of the particles, $\rho(t, x)$ is the associated charge density, and $E(t, x)$ is the electric field generated by the charged particles. \\
\\
Also, by taking relativistic effects into account, we consider the relativistic Vlasov-Poisson (RVP) system:
\begin{equation}
    \begin{aligned}
        &\partial_t f + \hat{v} \cdot \triangledown f + E \cdot \triangledown_v f = 0, \\ 
        &\rho(t, x) = \int_{\mathbb{R}^3} f(t, x, v) dv,\\
        &E(t, x) = \int_{\mathbb{R}^3} \frac{x-y}{| x-y |^3}\rho (t, y) dy,
    \end{aligned}
\end{equation}
while
\[
\hat{v} = \frac{v}{\sqrt{1+|v|^2}}
\]
 is the relativistic velocity.\\
To complete the description of the systems, we consider the particle distribution at time zero, denoted as 
\[
f_0(x, v) = f(0, x, v) \ge 0.
\]
Then an important property of both the VP and RVP system, which will be utilized throughout the paper, is the conservation of total mass \cite{The Cauchy problem in kinetic theory}, namely
\[
M(t) = \int \int_{\mathbb{R}^6} f(t, x, v) dvdx = \int \int_{\mathbb{R}^6} f_0(x, v) dvdx.
\]
\hfill\break
In this paper, we assume both systems to be spherically symmetric. It is an established result that for both systems, spherically symmetric initial data gives rise to global-in-time, spherically symmetric solutions (see \cite{Symmetric plasmas and their decay}, \cite{Propogation of moments and regularity for the three dimensional Vlasov-Poisson system}, and \cite{Global classical solution of the Vlasov-Poisson system in three dimensions for general initial data}). Moreover, \cite{Symmetric plasmas and their decay} proves that the charge density and electric field generated by spherically symmetric solutions of both systems eventually decay in the $L^\infty$ sense. Meanwhile, in the paper \cite{Arbitrarily large solutions of the Vlasov-Poisson system} and \cite{Concentrating solutions of the relativistic Vlasov- Maxwell system} by J. Ben-Artzi, S. Calogero and S. Pankavich, it is shown that for both systems, one can construct spherically symmetric solutions with arbitrarily small initial density and electric field, but over intermediate timescales, the particles will concentrate to create arbitrarily large density and field. The result is stated as follows:
\begin{theorem}{(J. Ben-Artzi, S. Calogero and S. Pankavich)}
    For any constants $C_1$ and $C_2$, there exists a smooth, spherically symmetric of the Vlasov-Poisson or the relativistic Vlasov-Poisson system such that
    \[
    \lVert \rho (0) \rVert_{\infty}, \quad  \lVert E (0) \rVert_{\infty} \le C_1
    \]
    but for some time $T>0$,
    \[
    \lVert \rho (T) \rVert_{\infty}, \quad  \lVert E (T) \rVert_{\infty} \ge C_2.
    \]
\end{theorem}
This result is in fact consistent with the decay estimate given by E. Horst in \cite{Symmetric plasmas and their decay}, showing that we can construct cases of focusing solutions in a short time period, even though they will eventually decay.\\
\\
However, one may notice that in both papers \cite{Arbitrarily large solutions of the Vlasov-Poisson system} and \cite{Concentrating solutions of the relativistic Vlasov- Maxwell system}, the focusing solutions are constructed in such a way that particles are concentrated closely around the origin at time $T$. In the case of the VP system, \cite{Arbitrarily large solutions of the Vlasov-Poisson system} requires $\mathscr{R}(T)$, the spatial radius of the particles at time $T$, to satisfy $\mathscr{R}(T) \le 100 \epsilon^2$, where $\epsilon > 0$ is a sufficiently small constant. Similarly in \cite{Concentrating solutions of the relativistic Vlasov- Maxwell system}, the solution of the RVP system satisfies $\mathscr{R}(T) \le 20 \epsilon$ for small $\epsilon$. This is because the methodology applied in both papers is to constraint the particles within a thin shield around the origin at time $T$. Since the supremums of both charge density and electric field are inversely related to the volume of the space filled with particles, this method allows the solutions to achieve arbitrarily large $\lVert \rho (T) \rVert_{\infty}$ and $\lVert E (T) \rVert_{\infty}$.\\
\\
Naturally, one may raise the following question: can a similar result be established by having the particles concentrate within an arbitrary given shield, but not just close to the origin? Since the shield is randomly away from the origin, the particles are more scatteringly distributed within a larger spatial radius, raising the difficulty of whether we can still achieve the concentrating effect in charge density and electric field. But ideally this can be achieve by having the particles concentrate within a thin enough shield. This paper aims at justifying this idea, namely proving that we can in fact construct focusing solutions for both the VP and RVP systems that have particles initially distributed in any given range away form the origin and concentrated within any targeting shield; and such solutions possess a similar concentrating behavior in charge density and electric field. Specifically, this paper aims at proving the following result: 
\begin{theorem}
    For any constant $C_1$, $C_2 > 0$ and $c > b > a > 0$, there exists a smooth, spherically symmetric solution of the 1) Vlasov-Poisson system or 2) relativistic Vlasov-Poisson system such that for any particle in the initial data, if we denote its spatial radius at time $t\ge 0$ by $\mathscr{R}(t)$, then
    \[
    \mathscr{R}(0) \ge c, \quad \text{and} \quad a \le \mathscr{R}(T) \le b
    \]
    for some time $T>0$.\\
    Moreover, we have the estimation that
    \[
    \lVert \rho (0) \rVert_{\infty}, \quad  \lVert E (0) \rVert_{\infty} \le C_1
    \]
    but for some time $T>0$,
    \[
    \lVert \rho (T) \rVert_{\infty}, \quad  \lVert E (T) \rVert_{\infty} \ge C_2
    \]
\end{theorem}

Theorem 2 is the result of combining Theorem 6 about the VP system and Theorem 7 about the RVP system, which we will prove in Section 3 and 4 respectively. In the following Section 2, I will first provide some background information concerning the spherical symmetry assumption we made and establish several useful lemmas. Then in Section 3, I will construct a focusing solution for the VP system, by first displaying our settings for parameters and initial data, and then prove that the solution satisfies the theorem. Finally in Section 4, I will construct a focusing solution for the RVP system with similar methodology but different data and details. 
\bigskip
\section*{2. Backgrounds and Lemmas}
\subsection*{2.1. Spherical Symmetry}
In this paper, we assume both the VP and RVP systems to be spherical symmetric. In particular, for both systems, we define the spatial radius $r$, inward velocity $w$, and square of the angular momentum $l$ of particles at initial time as follows:
\[
r = |x|, \quad w = \frac{x \cdot v}{r}, \quad l = |x \times v|^2.
\]
Let $f = f(t, r, w, l)$ be the density distribution of the particles under the spherically symmetric setting. Then using change of variable $(x, v) \rightarrow (r, w, l)$ (see \cite{Arbitrarily large solutions of the Vlasov-Poisson system} and \cite{Focusing solutions of the Vlasov-Poisson System} for more details), we can reduce the Vlasov-Poisson system to:
\begin{equation} \label{VP-spherical-symmetry}
    \partial_t f + w \partial_r f + (\frac{l}{r^3} + \frac{m(t, r)}{r^2}) \partial_w f = 0.
\end{equation}
Similarly (see \cite{Concentrating solutions of the relativistic Vlasov- Maxwell system}), the relativistic Vlasov-Poisson system can be reduced to:
\begin{equation} \label{RVP-spherical-symmetry}
    \partial_t f + \frac{w}{\sqrt{1+w^2 + lr^{-2}}}\partial_r f + \left( \frac{l}{r^3\sqrt{1+w^2 + lr^{-2}}} + \frac{m(t, r)}{r^2} \right) \partial_w f = 0.
\end{equation}
In both equations, we know
\begin{equation} \label{mass-spherical-symmetry}
    m(t, r) = 4 \pi \int^r_0 s^2 \rho(t, s) ds,
\end{equation}
and 
\begin{equation} \label{density-spherical-symmetry}
    \rho(t, r) = \frac{\pi}{r^2} \int^\infty_0 \int^\infty_{-\infty} f(t, r, w, l) dw \, dl.
\end{equation} 
Also, the electric field is given by 
\begin{equation} \label{electric-field-spherical-symmetry}
    E(t, x) = \frac{m(t, r)}{r^2} \frac{x}{r}
\end{equation}
in both systems.\\
The total mass of the particles is 
\begin{equation}
    M = 4\pi^2 \int^\infty_0 \int^\infty_{-\infty}\int^\infty_0 f_0(r, w, l) dldwdr,
\end{equation}
which, as stated before, is conserved. 
\\
\\
Following these descriptions, we can express the characteristics of the Vlasov-Poisson system in such form:
\begin{equation} \label{VP-characteristics}
    \begin{aligned}
        &\frac{d}{ds}\mathscr{R}(s) = \mathscr{W}(s),\\
        &\frac{d}{ds}\mathscr{W}(s) = \frac{\mathscr{L}(s)}{\mathscr{R}(s)^3} + \frac{m(s, \mathscr{R}(s))}{\mathscr{R}(s)^2}, \\
        &\frac{d}{ds}\mathscr{L}(s) = 0,
    \end{aligned}
\end{equation}
with the initial conditions
\begin{equation} \label{VP-characteristics-initial}
    \mathscr{R}(0) = r, \quad \mathscr{W}(0) = w, \quad \mathscr{L}(0) = l.
\end{equation}
Since the angular momentum of particles is conserved in time on the support of $f(t)$, we have that $\mathscr{L}(s) = l$ for all $s \ge 0$. \\
\\
Similarly, we can express the characteristics of the relativistic Vlasovi-Poisson system as:
\begin{equation} \label{RVP-characteristics}
    \begin{aligned}
        &\frac{d}{ds}\mathscr{R}(s) = \frac{\mathscr{W}(s)}{\sqrt{1+\mathscr{W}(s)^2 + \mathscr{L}(s) \mathscr{R}(s)^{-2}}}\\
        &\frac{d}{ds}\mathscr{W}(s) = \frac{\mathscr{L}(s)}{\mathscr{R}(s)^3\sqrt{1+\mathscr{W}(s)^2 + \mathscr{L}(s) \mathscr{R}(s)^{-2}}} + \frac{m(s, \mathscr{R}(s))}{\mathscr{R}(s)^2}, \\
        &\frac{d}{ds}\mathscr{L}(s) = 0,
    \end{aligned}
\end{equation}
which is also subject to the initial conditions 
\begin{equation} \label{RVP-characteristics-initial}
    \mathscr{R}(0) = r, \quad \mathscr{W}(0) = w, \quad \mathscr{L}(0) = l.
\end{equation}
and conservation of angular momentum as well.\\
\\
Finally, for both systems, we denote
\[
S(t) = \{ (r, w, l): f(t, r, w, l) > 0 \},
\]
for $t \ge 0$. In particular, we have 
\[
S(0) = \{ (r, w, l): f_0(r, w, l) > 0 \}.
\]
\subsection*{2.2. Lemmas}
We now give some lemmas that are useful for our proof of the theorem later on. Most parts of these lemmas are established by Jonathan Ben-Artzi, Simone Calogero, and Stephen Pankavich in \cite{Arbitrarily large solutions of the Vlasov-Poisson system} and \cite{Concentrating solutions of the relativistic Vlasov- Maxwell system}. \\
\\
The first lemma mainly concerns with the behavior of the characteristics of the VP system as described by the ODE system in \ref{VP-characteristics} and \ref{VP-characteristics-initial}. In particular, we want to derive an estimate for the particle spatial radius at given time $T>0$ using the initial parameters $r$, $w$, $l$, total mass $M$, and the time $T$.
\begin{lemma}
    Let $l > 0$, $M \ge 0$, $r > 0$, and $w < 0$ be given. Also, let $(\mathscr{R}(t),\, \mathscr{W}(t), \, \mathscr{L}(t))$ be a solution to \ref{VP-characteristics} and \ref{VP-characteristics-initial} for all $t \ge 0$. Then \\
    (1). There exists a unique $T_0 > 0$ such that $\mathscr{W}(t) < 0$ for $t \in [0,\, T_0)$, $\mathscr{W} (T_0) = 0$, and $\mathscr{W}(t) >0$ for $t \in (T_0,\, \infty)$.\\
    (2). $T_0$ satisfies
    \[
    T_0 > \frac{r}{|w|} \left(1- \frac{\sqrt{l+Mr}}{r|w|} \right) = \frac{r}{|w|}- \frac{\sqrt{l+Mr}}{w^2}.
    \]
    (3). For all $t \in [0, \, T_0)$, we have
    \[
    \mathscr{R}(t)^2 \le (r+wt)^2 + (lr^{-2} + Mr^{-1})t^2.
    \]
    (4). Let $0 < T < T_0$. Then $\mathscr{R}(T)$ satisfies
    \[
    \mathscr{R}(T) \ge \frac{l}{2}r^{-3}T^2 + wT + r.
    \]
\end{lemma}
\begin{proof}
    One can find the detailed proof of (1), (2), and (3) from the lemma 3 and 4 in \cite{Arbitrarily large solutions of the Vlasov-Poisson system}. Thus we only prove (4) here. (1) indicates that for $t \le T_0$, $\mathscr{W}(t) = \dot{\mathscr{R}}(t) \le 0$. Therefore, choose time $T$ such that $0 < T < T_0$, then we know that $\mathscr{R}(T) \le \mathscr{R}(0) = r$. Substituting this result into \ref{VP-characteristics} gives
\[
\ddot{\mathscr{R}}(T) \ge l\mathscr{R}(T)^{-3} \ge l\mathscr{R}(0)^{-3}.
\]
Thus we can further deduce that
\[
\dot{\mathscr{R}}(T) \ge lR(0)^{-3}T + \dot{\mathscr{R}}(0),
\]
\[
\mathscr{R}(T) \ge \frac{l}{2} \mathscr{R}(0)^{-3}T^2 + \dot{\mathscr{R}}(0)T + \mathscr{R}(0),
\]
which gives the result.
\end{proof}
The second lemma is an analogue of Lemma 3 for the ODE system associated with the RVP system, as described in \ref{RVP-characteristics} and \ref{RVP-characteristics-initial}.
\begin{lemma}
    Let $l > 0$, $M \ge 0$, $r > 0$, and $w < 0$ be given. Also, let $(\mathscr{R}(t),\, \mathscr{W}(t), \, \mathscr{L}(t))$ be a solution to \ref{RVP-characteristics} and \ref{RVP-characteristics-initial} for all $t \ge 0$, and define
    \[
    D = l + Mr\sqrt{1+w^2 + lr^{-2}}.
    \]
    Then \\
    (1). There exists a unique $T_0 > 0$ such that $\mathscr{W}(t) < 0$ for $t \in [0,\, T_0)$, $\mathscr{W} (T_0) = 0$, and $\mathscr{W}(t) >0$ for $t \in (T_0,\, \infty)$.\\
    (2). $T_0$ satisfies the bounds
    \[
    r \left(1 - \sqrt{\frac{D}{r^2w^2+D}} \right) \le T_0 \le \frac{-wr^3\sqrt{1+w^2+lr^{-2}}}{l}.
    \]
    (3). For all $t \in [0, \, T_0)$, we have
    \[
    \mathscr{W}(t)^2 + l\mathscr{R}(t)^{-2} \le w^2 + lr^{-2}.
    \]
    (4). For all $t \in [0, \, T_0)$, we have
    \[
    \mathscr{R}(t)^2 \le \left(r- \frac{|w|}{\sqrt{1+w^2+lr^{-2}}}t \right)^2 + \frac{D}{r^2(1+w^2+lr^{-2})}t^2.
    \]
    (5). Let $0 < T < T_0$. Then $\mathscr{R}(T)$ satisfies
    \[
    \mathscr{R}(T) \ge \frac{l}{2r^3(1+w^2 + lr^{-2})}T^2 + \frac{w}{\sqrt{1+w^2+lr^{-2}}}T + r.
    \]
\end{lemma}
\hfill\break
\begin{proof}
    One can find the detailed proof for (1) to (4) from lemma 3 in \cite{Concentrating solutions of the relativistic Vlasov- Maxwell system}, so we prove for (5) here. (3) informs us that 
    \[
    \frac{d}{ds}\mathscr{R}(s) = \frac{\mathscr{W}(s)}{\sqrt{1+\mathscr{W}(s)^2 + \mathscr{L}(s) \mathscr{R}(s)^{-2}}} \ge \frac{\mathscr{W}(s)}{\sqrt{1+w^2+lr^{-2}}},
    \]
    and 
    \[
    \frac{d}{ds}\mathscr{W}(s) \ge \frac{\mathscr{L}(s)}{\mathscr{R}(s)^3\sqrt{1+\mathscr{W}(s)^2 + \mathscr{L}(s) \mathscr{R}(s)^{-2}}}
    \ge \frac{\mathscr{L}(s)}{r^3 \sqrt{1+w^2+lr^{-2}}}
    \]
    Thus
    \[
    \mathscr{W}(T) \ge \frac{l}{r^3\sqrt{1+w^2+lr^{-2}}}T +w.
    \]
    Then we have
    \[
    \frac{d}{ds}\mathscr{R}(T) \ge \frac{l}{r^3(1+w^2 + lr^{-2})}T + \frac{w}{\sqrt{1+w^2+lr^{-2}}}
    \]
    and finally
    \[
    \mathscr{R}(T) \ge \frac{l}{2r^3(1+w^2 + lr^{-2})}T^2 + \frac{w}{\sqrt{1+w^2+lr^{-2}}}T + r.
    \]
\end{proof}
The last lemma provides us with a lower bound for the $L^\infty$ norm of the particle charge density and the electric field at time $t >0$, which can be applied to both the VP and RVP system.
\begin{lemma}
    Let $f(t, r, w, l)$ be a spherically symmetric solution of the VP or RVP system with associated charge density $\rho(t, r)$ and electric field $E(t, x)$. \\
    Also, let $(\mathscr{R}(t, 0, r, w, l),\, \mathscr{W}(t, 0, r, w, l), \, \mathscr{L}(t, 0, r, w, l))$ be a characteristic solution of \ref{VP-characteristics}. If at some $T \ge 0$ we have 
    \[
    a \le \sup_{(r, w, l)\in S(0)} \mathscr{R}(T, 0, r, w, l) \le b,
    \]
    then
    \[
    \lVert \rho (T) \rVert_{\infty} \ge \frac{3M}{4\pi(b^3 - a^3)},
    \]
    and 
    \[
    \lVert E(T) \rVert_{\infty} \ge \frac{M}{b^2}.
    \]
\end{lemma}
\hfill\break
\begin{proof}
    Please see lemma 5 in \cite{Arbitrarily large solutions of the Vlasov-Poisson system} or lemma 4 in \cite{Concentrating solutions of the relativistic Vlasov- Maxwell system}. 
\end{proof}
\bigskip

\section*{3.Focusing Solution of the Vlasov-Poisson System}
In this section we attempt to establish the following result:
\begin{theorem}
    For any constant $C_1$, $C_2 > 0$ and $c > b > a > 0$, there exists a smooth, spherically symmetric solution of the Vlasov-Poisson system such that for $(r, w, l) \in S(0)$ we have 
    \[
    r \ge c,\quad  \text{and} \quad a \le \mathscr{R}(T, 0, r, w, l) \le b
    \]
    for some time $T>0$.\\
    Moreover, we have the estimation that
    \[
    \lVert \rho (0) \rVert_{\infty}, \quad  \lVert E (0) \rVert_{\infty} \le C_1
    \]
    and
    \[
    \lVert \rho (T) \rVert_{\infty}, \quad  \lVert E (T) \rVert_{\infty} \ge C_2
    \]
\end{theorem}

\subsection*{3.1. Parameters and Initial Data Settings}
We now begin the construction of the focusing solution. Given constants $C_1,\, C_2 >0$ and $c>b>a>0$, we set a few parameters and choose a smooth, spherically symmetric solution of the VP system that satisfies the following conditions:\\
\begin{enumerate}[label=(\roman*)]
    \item Define constant $k = (1+\frac{a^2}{b^2})\div2 = \frac{a^2+b^2}{2b^2}$, and $\epsilon = \frac{\sqrt{k}b-a}{8} = \frac{\sqrt{a^2+b^2}-\sqrt{2}a}{8\sqrt{2}}$. Also, choose $a_0 >0$ large enough such that
\begin{equation}\label{a_0-large-enough}
\begin{aligned}
    a_0 \ge \text{max} \{ &\frac{\sqrt{a^2+b^2}-\sqrt{2}a}{8\sqrt{2}}+c,\,
    C_1^{-1},\,
    3\sqrt{2}\pi(\sqrt{a^2+b^2}-\sqrt{2}a)C_1^{-1},  \,\\ 
    &2^{\frac{-11}{6}}\pi^{\frac{1}{3}}(\sqrt{a^2+b^2}-\sqrt{2}a)C_1^{-\frac{1}{3}},\,
    \frac{16\sqrt{2}(b^3-a^3)C_2}{3(\sqrt{a^2+b^2}-\sqrt{2}a)},\,
    \frac{4\sqrt{2}b^2C_2}{\pi (\sqrt{a^2+b^2}-\sqrt{2}a)}
    \}.        
\end{aligned}
\end{equation}

\item For $(r, w, l) \in S(0)$, we have
\begin{equation}\label{initial-radius}
   a_0 - \epsilon < r < a_0 + \epsilon.
\end{equation}
Moreover, the initial charge density $\rho_0$ satisfies
\begin{equation}\label{initial-density-upper-bound}
    \rho_0(r) \le \frac{1}{a_0} \quad \text{for all } r >0,
\end{equation}
and 
\begin{equation} \label{initial-density-lower-bound}
    \rho_0(r) = \frac{1}{2a_0} \quad \text{for } r \in [a_0 - \frac{1}{2}\epsilon, \, a_0 + \frac{1}{2}\epsilon].
\end{equation}

\item We define time $T>0$ small enough such that 
\begin{equation} \label{T-setting}
    \begin{aligned}
        T \le \text{min}\{&\frac{a(|\sqrt{k}b-a_0-\epsilon|)}{\sqrt{l+(2\pi a_0 \epsilon + \frac{\pi}{6}a_0^{-1}\epsilon)(a_0-\epsilon)}}, \, \\
        &\sqrt{\frac{1-k}{l(a_0-\epsilon)^{-2}+(8\pi a_0\epsilon + \frac{8\pi}{3}a_0^{-1}\epsilon^3)(a_0-\epsilon)^{-1}}}b\,
        \}.
    \end{aligned}
\end{equation}

\item Finally, for $(r, w, l) \in S(0)$, we let 
\begin{equation} \label{initial-inward-velocity}
    \frac{a-a_0 + \epsilon}{T} \le w \le \frac{\sqrt{k}b - a_0 - \epsilon}{T},
\end{equation}
and we require $l>0$.
\end{enumerate}

\begin{remark}
    In order for the inequality in \ref{initial-inward-velocity} to be valid, we need to check $a-a_0 + \epsilon \le \sqrt{k}b - a_0 - \epsilon $, namely, 
    \[
    \epsilon < \frac{\sqrt{k}b-a}{2}.
    \]
    Our choice of constant $\epsilon$ can verify this point.
\end{remark}
\hfill\break
\begin{remark}
    To see that these choices of parameters are feasible, we now provide an example of a smooth function that satisfy all the above conditions.\\
    \\
    First of all, based on the given constants $C_1$, $C_2$, $a$, $b$, and $c$, we first fix constants $k$, $\epsilon$, and $a_0$ that satisfy the requirements in (i), and choose constant $T>0$ that it satisfies (iii).\\
    Now we choose $H: \, [0, \infty) \rightarrow [0, \infty)$ to be a smooth function that satisfies 
    \[
    \int_{\mathbf{R}^3} H(|u|^2) du = \frac{T^3}{2a_0}
    \]
    with $supp(H) \subset [0, 1]$. We then re-scale this function with a scalar $\delta > 0$ by defining
    \[
    H_\delta (|u|^2) = \frac{1}{\delta^3} H(\frac{|u|^2}{\delta^2}).
    \]
    Then we will find that 
    \[
    \int_{\mathbf{R}^3} H_\delta(|u|^2) du = \frac{T^3}{2a_0}
    \]
    and $supp(H_\delta) \subset [0, \delta^2]$. \\
    \\
    Next, we choose a smooth cut-off function $\chi \in C^\infty ((0, \infty), [0, 1])$ such that 
    \[
    \begin{cases} 
      \chi(r) = 1 &\text{for   } \quad a_0 -\frac{1}{2}\epsilon \le |r| \le a_0 + \frac{1}{2}\epsilon \\
      \chi(r) = 0 & \text{for   } \quad|r| \le a_0 - \epsilon \text{ or } |r| \ge a_0 + \epsilon \\
   \end{cases}
    \]
    We then fix constants $d = \frac{a+\sqrt{k}b}{2}$, $\delta = \frac{\sqrt{k}b-a - 4\epsilon}{2} = \frac{\sqrt{k}b-a}{4}> 0$, and define function
    \[
    f_0(x, v) = H_\delta (|\left(1-\frac{d}{|x|} \right)x + Tv|^2) \chi(|x|).
    \]
    Clearly $f_0$ is smooth, by claiming that it satisfies the requirements in (ii) and (iv) we finish our construction of the solution.\\
    \\
    To verify this, we first notice that $r = |x|$, so that we can obtain from the definition of $\chi$ that 
    \[
    a_0 - \epsilon \le r \le a_0 + \epsilon.
    \]
    Then an immediate observation is that $ d = \frac{a+ \sqrt{k}b}{2} \le a_0 - \epsilon \le r$ by our choice of parameters in (i). Thus we know $0 < 1-\frac{d}{|x|}<1$. \\
    Next, the support of $h_\delta$ indicates us that
    \[
    |\left(1-\frac{d}{|x|} \right)x+Tv|^2 \le \delta^2.
    \]
    Using change of variables and the relation $|v|^2 = w^2 + lr^{-2}$, we can obtain that 
    \[
    (r -d + Tw)^2 + T^2lr^{-2} \le \delta^2,
    \]
    and thus $|r-d+ Tw| \le \delta$. Expanding this inequality informs us that
    \[
    \frac{d - \delta -r}{T} \le w \le \frac{d + \delta -r}{T}.
    \]
    We've set $d$ and $\delta$ such that $d - \delta = a + 2\epsilon$ and $d + \delta = \sqrt{k}b- 2\epsilon$. Moreover, we know that $a_0 - \epsilon \le r \le a_0 + \epsilon$. Therefore we can deduce that 
    \[
    \frac{a+ 2\epsilon - (a_0 + \epsilon)}{T} \le \frac{d-\delta - r}{T}  \le w \le \frac{d+ \delta -r}{T}\le \frac{\sqrt{k}b -2\epsilon - (a_0 - \epsilon)}{T}.
    \]
    which satisfies the requirement in (iv).\\
    \\
    Finally, to check that the solution satisfies (ii), we notice that 
    \[
    \int_{\mathbf{R}^3} H_\delta (|\left(1-\frac{d}{|x|} \right)x + Tv|^2) dv = \frac{1}{T^3} \int_{\mathbf{R}^3} H_\delta (|u|^2)du = \frac{1}{2a_0}.
    \]   
    and thus we can calculate 
    \[
    \rho_0 (x) = \int_{\mathbf{R}^3} f_0 (x, v) = \left( \int_{\mathbf{R}^3} H_\delta (|\left(1-\frac{d}{|x|} \right)x + Tv|^2) dv \right) \chi(|x|)  = \frac{1}{2a_0} \chi(|x|),
    \]
    which satisfies the requirements in (ii).\\
    Therefore, we've verified $f_0(x, v)$ to be a smooth function that satisfies all the settings for the initial data.
\end{remark}
\begin{remark}
    The choice of $a_0$ might seem complicated, but using the expression of orders we can see how it's connected to other given constants.
    \begin{itemize}
        \item $a_0 \sim O(C_1^{-1})$. $C_1$ is the required upper bound for $\lVert \rho (0) \rVert_{\infty}$ and $\lVert E(0) \rVert_{\infty}$. $a_0$ being inversely related to $C_1$ indicates that at time zero, if we want the charge density and electric field to be small, then the particles should be placed further from the origin, so that they're distributed more scatteringly.
        \item $a_0 \sim O(C_2)$. $C_2$ is the required lower bound for $\lVert \rho (T) \rVert_{\infty}$ and $\lVert E(T) \rVert_{\infty}$. $a_0$ being positively relating to $C_2$ indicate that larger concentrating density and electric field requires larger spatial radius of particle distribution at time zero. This is generally because such distribution allows a larger total mass of particles, which makes a stronger concentration effect possible.
        \item $a_0 \ge \frac{8\sqrt{2}(b^3-a^3)C_2}{3(\sqrt{a^2+b^2}-\sqrt{2}a)} \ge \frac{2(b^3- a^3) C_2}{3 \epsilon}$, showing that $a_0$ is also positively related to the thickness and volume of the targeting shield of concentration, which is $\sim O(b^3 - a^3)$.
    \end{itemize}
Moreover, we can use the conditions on $a_0$ to derive an estimation on the scale and order of other parameters, such as the total particle mass $M$ and the initial charge density $\rho_0$. In fact, these estimates are the strongest motivation behind our complicated choice of $a_0$, as we will see how they help with our calculation and proof later on.  
    \begin{itemize}
      \item \ref{initial-density-lower-bound} and \ref{initial-density-upper-bound} indicates that $\rho_0 \sim O(a_0^{-1})$. The inverse relationship between $\rho_0$ and $a_0$ not only helps us meet the upper bound requirement on $ \lVert \rho (0) \rVert_{\infty}$, but also help us control the order of the total particle mass. 
\item An immediate result following the relationship between $\rho_0$ and $a_0$, and the inequailties \ref{initial-radius}, \ref{initial-density-upper-bound}, and \ref{initial-density-lower-bound}, is that we can derive a range for the total particle mass $M$:
    \begin{equation} \label{mass-upper-bound}
    \begin{aligned}
        M &= \int_{\mathbb{R}^3}\rho_0(x)dx = 4\pi \int^{a_0+\epsilon}_{a_0-\epsilon}\rho_0(r) r^2dr
        \le \frac{4\pi}{a_0}\int^{a_0+\epsilon}_{a_0-\epsilon}\rho_0(r) r^2dr\\
        &\le 8\pi a_0\epsilon + \frac{8\pi}{3}a_0^{-1}\epsilon^3,
    \end{aligned}
\end{equation}
    and 
    \begin{equation} \label{mass-lower-bound}
    \begin{aligned}
        M &\ge 4\pi \int^{a_0+\frac{1}{2}\epsilon}_{a_0-\frac{1}{2}\epsilon}\rho_0(r)r^2dr = 
        \frac{2\pi}{a_0}\int^{a_0+\frac{1}{2}\epsilon}_{a_0-\frac{1}{2}\epsilon}r^2dr\\
        &= 2\pi a_0 \epsilon + \frac{\pi}{6}a_0^{-1}\epsilon^3.                
    \end{aligned}
\end{equation}
      
        \item Finally, \ref{initial-inward-velocity} tells us that \[
        a - a_0 + \epsilon \le a-r \le wT \le\sqrt{k}b-r \le \sqrt{kb}-a_0-\epsilon.
        \]
        $r + wT$ is a rough estimation of the particle position at time $T$. Thus our choice of $T$ in \ref{T-setting} and $w$ in \ref{initial-inward-velocity} is meant to control the particle position at time of concentration. 
    \end{itemize}
\end{remark}

\subsection*{3.2. Proof of Theorem 6}
\begin{proof}
Our strategy is to set the particles far away from the targeting shield at time zero, and endow them with large enough inward velocity. Then by balancing between the initial inward velocity $w$ and the time $T$, we can make the particles concentrate within the targeting shield in a short time period while they're still moving inwards. Since we set the initial particle spatial radius $r$ to be large enough, the particles are scattering at time zero and have low charge density and electric field. Meanwhile, the total particle mass $M$ is controlled to be large enough, so that both the charge density and electric field will rise drastically once the particles concentrate at time $T$.\\
\\
We first check the particles' distributing region at time zero. For $(r, w, l) \in S(0)$, we have 
\[
r > a_0 - \epsilon = a_0 - \frac{\sqrt{a^2+b^2}-\sqrt{2}a}{8\sqrt{2}}> c
\]
by our choice of $a_0$. Thus the particles are initially distributed in the given region. \\
Next, we check that the initial charge density and electric field are small enough. At time zero, we notice that by \ref{initial-density-upper-bound}, 
\[
\lVert \rho (0) \rVert_{\infty} \le \frac{1}{a_0} \le C_1.
\]
Moreover, using the conditions in \ref{electric-field-spherical-symmetry} and \ref{mass-upper-bound}, we can derive an upper bound for the initial electric field in each region. 
\begin{enumerate}
    \item For $|x| < a_0-\epsilon$, since the particles have initial spatial radius $r \ge a_0 - \epsilon$, this region is empty and we know that  $|E(0, x)| = 0$. 
    \item For $|x| > a_0+ \epsilon$, we know from \ref{electric-field-spherical-symmetry} that
\[
|E(0, x)| \le \frac{M}{r^2} \le \frac{8\pi a_0\epsilon + \frac{8\pi}{3}a_0^{-1}\epsilon^3}{(a_0+\epsilon)^2} \le 8\pi a_0^{-1}\epsilon + \frac{8\pi}{3}a_0^{-3}\epsilon^3.
\]
\item Lastly, for $a_0 -\epsilon \le |x| \le a_0 + \epsilon$, we have
\[
|E(0, x)| \le \frac{M}{r^2} \le \frac{8\pi a_0\epsilon + \frac{8\pi}{3}a_0^{-1}\epsilon^3}{(a_0-\epsilon)^2} \le 32\pi a_0^{-1}\epsilon + \frac{32\pi}{3}a_0^{-3}\epsilon^3.
\]
\end{enumerate}
Therefore, an uniform upper bound for the initial electric field is 
\[
\lVert E(0) \rVert_{\infty} \le 32\pi a_0^{-1}\epsilon + \frac{32\pi}{3}a_0^{-3}\epsilon^3,
\]
and we want to show that this value is bounded by $C_1$. Since this value is inversely related to $a_0$, we should be able to do so by choosing $a_0$ large enough. Specifically, we've chosen $a_0 \ge \text{max}\{ 3\sqrt{2}\pi(\sqrt{a^2+b^2}-\sqrt{2}a)C_1^{-1},  \, 2^{\frac{-11}{6}}\pi^{1\over 3}(\sqrt{a^2+b^2}-\sqrt{2}a)C_1^{-\frac{1}{3}} $. Then we can calculate an upper bound for the two terms separately that 
\[
32\pi a_0^{-1}\epsilon = 2\sqrt{2}\pi(\sqrt{a^2+b^2}-\sqrt{2}a)a_0^{-1} \le \frac{2}{3}C_1
\]
and 
\[
\frac{32\pi}{3}a_0^{-3}\epsilon^3 = \frac{\sqrt{2}\pi}{192}(\sqrt{a^2+b^2}-\sqrt{2}a)^3a_0^{-3}\le \frac{1}{3}C_1.
\] 
Thus by adding these results together, we can get
\[
\lVert E(0) \rVert_{\infty} \le 32\pi a_0^{-1}\epsilon + \frac{32\pi}{3}a_0^{-3}\epsilon^3 \le C_1. 
\]
Therefore, we've proven that our solution satisfies the conditions on the initial particle spatial radius, charge density, and electric field.\\
\\
We now look at time $T$, which is the presumptive time for the particles to concentrate. Lemma 3 informs us that particles with inward initial velocity will first move inwards until a time $T_0$ and then move outwards. To have a more accurate control over the particle location,we want the concentration to happen in a short time period during which the particles are still moving inwards. This is to say, we want $T< T_0$. To meet this requirement, Lemma 3 implies that we will need
\begin{equation}\label{condition-for{T<T_0}}
    T \le \frac{r}{|w|}-\frac{\sqrt{l+Mr}}{w^2} \quad \text{or }\quad wT\ge -r + \frac{\sqrt{l+Mr}}{|w|}.
\end{equation}
Now, in \ref{T-setting} we've asked $T \le \frac{a(|\sqrt{k}b-a_0-\epsilon|)}{\sqrt{l+(2\pi a_0 \epsilon + \frac{\pi}{6}a_0^{-1}\epsilon)(a_0-\epsilon)}}$. Since for $(r, w, l) \in S(0)$, we've let $\frac{a-a_0 + \epsilon}{T} \le w \le \frac{\sqrt{k}b - a_0 - \epsilon}{T}$, we can bound $|w|$ by 
\[
|w| \ge \left| \frac{\sqrt{k}b-a_0-\epsilon}{T}\right| \ge \frac{\sqrt{l+(2\pi a_0 \epsilon + \frac{\pi}{6}a_0^{-1}\epsilon)(a_0-\epsilon)}}{a} \ge \frac{\sqrt{l+Mr}}{a}.
\]
Therefore,
\[
\frac{\sqrt{l+Mr}}{|w|} \le a = (a-a_0+\epsilon) + (a_0 - \epsilon) \le wT +r,
\]
so that the inequality in \ref{condition-for{T<T_0}} is satisfied, and we've established $T \le T_0$ as we want.\\
\\
Now that the concentration is proven to happen in a short time period while the particles are moving inwards, we want to further verify that the particles are located within the targeting shield at time $T$. Specifically, we want that for all $(r, w, l) \in S(0)$, 
\[
a \le \mathscr{R}(T, 0, r, w, l) \le b.
\]
To do so, we can utilize the lower and upper bound for $\mathscr{R}(T)$ we've developed in Lemma 3. Specifically, we will need
\begin{equation} \label{R(T)>=a}
    \frac{l}{2}r^{-3}T^2 + wT + r \ge a
\end{equation}
and 
\begin{equation} \label{R(T)<=b}
     (r+wT)^2 + (lr^{-2} + Mr^{-1})T^2 \le b^2.
\end{equation}
Notice that the term $wT$ appears in both inequalities. A physical interpretation is that this term, the product of initial inward velocity and time of concentration, is a rough estimation of the distance the particles travelled inward in a short time period $T$. Thus controlling this term will help us satisfy both inequalities \ref{R(T)>=a} and \ref{R(T)<=b}. The range we choose for $w$  in \ref{initial-inward-velocity} tells us that
\[
a-a_0+\epsilon \le wT \le \sqrt{k}b - a_0 - \epsilon.
\]
Then immediately we have 
\[
\frac{l}{2}r^{-3}T^2 + wT + r \ge wT+r \ge (a-a_0+\epsilon) + (a_0 -\epsilon) =a,
\]
so that \ref{R(T)>=a} is satisfied. \\
Similarly, since $  wT \le \sqrt{k}b - a_0 - \epsilon $, we have $r+wT \le \sqrt{k}b$, so that
\[
(r+wT)^2 \le kb^2.
\]
Then in order to verify \ref{R(T)<=b}, we only need to show that the remaining term on the left side, $(r^{-2} + Mr^{-1})T^2$, is smaller than $b^2 - kb^2$. Since we have estimated $r$ and $M$ using constants $a_0$ and $\epsilon$, this inequality can be achieved simply by setting $T$ small enough. Specifically, in \ref{T-setting} we've set $T \le \sqrt{\frac{1-k}{l(a_0-\epsilon)^{-2}+(8\pi a_0\epsilon + \frac{8\pi}{3}a_0^{-1}\epsilon^3)(a_0-\epsilon)^{-1}}}b$. Plug this value into the term $(r^{-2} + Mr^{-1})T^2$ gives
\[
(r^{-2}+Mr^{-1})T^2 \le \left( l(a_0-\epsilon)^{-2}+(8\pi a_0\epsilon + \frac{8\pi}{3}a_0^{-1}\epsilon^3)(a_0-\epsilon)^{-1} \right) T^2 \le (1-k)b^2
\]
Therefore, combining these two result, we have 
\[
(r+wT)^2 + (lr^{-2} + Mr^{-1})T^2 \le kb^2 + (1-k)b^2 =b^2,
\]
so that \ref{R(T)<=b} is confirmed. As we've verified both inequalities \ref{R(T)>=a} and \ref{R(T)<=b}, we can use Lemma 3 to confirm that for all $(r, w, l) \in S(0)$, we have $a \le \mathscr{R}(T) \le b$.\\
\\
Finally, we're left to verify the charge density and electric field at time $T$ is large enough. Since we've established that the particles will be concentrated within $a \le \mathscr{R}(T) \le b$ at time $T$, we can use Lemma 5 to infer that both $\lVert \rho(T) \rVert_\infty$ and $\lVert E(T) \rVert_\infty$ are inversely related to the total particle mass. More specifically, Lemma 5 tells us that in order for $\lVert \rho(T) \rVert_\infty$, $\lVert E(T) \rVert_\infty \ge C_2$, we will need
\[
M \ge max\{\frac{4\pi}{3}(b^3-a^3)C_2, \, b^2C_2\}.
\]
Since we can estimate $M$ with 
\[
2\pi a_0 \epsilon + \frac{\pi}{6}a_0^{-1}\epsilon^3 \ge 2\pi a_0\epsilon,
\]
we know that $M$ is positively related with $a_0$. Therefore, we can satisfy both conditions by setting $a_0$ large enough. In \ref{a_0-large-enough}, we've set $a_0 \ge \frac{1 6\sqrt{2}(b^3-a^3)C_2}{3(\sqrt{a^2+b^2}-\sqrt{2}a)}$. Plug this value into our estimation of $M$ gives
\[
M \ge 2\pi a_0\epsilon \ge \frac{(\pi \sqrt{a^2+b^2}-\sqrt{2}a)}{4\sqrt{2}}a_0 \ge \frac{4\pi}{3}(b^3-a^3).
\]
Similarly, we've asked $a_0 \ge \frac{4\sqrt{2}b^2C_2}{\pi (\sqrt{a^2+b^2}-\sqrt{2}a)}$, which provides us with
\[
M \ge 2\pi a_0\epsilon \ge \frac{(\pi \sqrt{a^2+b^2}-\sqrt{2}a)}{4\sqrt{2}}a_0 \ge b^2C_2.
\]
Thus by Lemma 5, $\lVert \rho(T) \rVert_\infty \ge C_2$ and $\lVert E(T) \rVert_\infty \ge C_2$ is verified. \\
\\
Till now, we've confirmed that the solution of the VP system we choose is initially located at region $|x| \ge c$, and concentration within the targeting shield $a \le |x| \le b$ at some time $T$. Moreover, we've verified that the solution satisfies the conditions on charge density and electric field at time zero and time $T$, namely $\lVert \rho(0) \rVert_\infty \le C_1$, $\lVert E(0) \rVert_\infty \le C_1$; and $\lVert \rho(T) \rVert_\infty \ge C_2$, $\lVert E(T) \rVert_\infty \ge C_2$. Therefore, Theorem 6 is proven.
\end{proof}

\bigskip
\section*{4. Focusing Solution of the Relativistic Vlasov-Poisson System}
In this section, we want to prove a similar result for the relativistic Vlasov-Poisson System, namely
\begin{theorem}
    For any constant $C_1$, $C_2 > 0$ and $c > b > a > 0$, there exists a smooth, spherically symmetric solution of the relativistic Vlasov-Poisson system such that for $(r, w, l) \in S(0)$ we have 
    \[
    r \ge c,\quad  \text{and} \quad a \le \mathscr{R}(T, 0, r, w, l) \le b,
    \]
    For some $T >0$.\\
    Moreover, we have the estimation that
    \[
    \lVert \rho (0) \rVert_{\infty}, \quad  \lVert E (0) \rVert_{\infty} \le C_1,
    \]
    but
    \[
    \lVert \rho (T) \rVert_{\infty}, \quad  \lVert E (T) \rVert_{\infty} \ge C_2.
    \]
\end{theorem}
\subsection*{4.1. Parameters and Initial Data Settings}
We now begin constructing the focusing solution of the RVP system, which shares certain degrees of similarity with the solution of the VP system in the previous section, but is in fact more complicated in term of choice of parameters and settings of initial data.\\
\\
Given constants $C_1,\, C_2 >0$ and $c>b>a>0$, we introduce the following parameters and choose a smooth, spherically symmetric solution of the RVP system that satisfies the following conditions:\\
\begin{enumerate}[label=(\roman*)]
\item Define constant $k = (1+\frac{a^2}{b^2})\div2 = \frac{a^2+b^2}{2b^2}$, and $\epsilon = \frac{\sqrt{k}b-a}{4} = \frac{\sqrt{a^2+b^2}-\sqrt{2}a}{4\sqrt{2}}$. Also, choose constant $a_0 >0$ large enough such that
\begin{equation}\label{a_0-large-enough}
\begin{aligned}
    a_0 \ge \text{max} \{ &\frac{\sqrt{a^2+b^2}-\sqrt{2}a}{4\sqrt{2}}+c,\,
    C_1^{-1},\,
    6\sqrt{2}\pi(\sqrt{a^2+b^2}-\sqrt{2}a)C_1^{-1},  \,\\ 
    &2^{\frac{5}{6}}\pi^{\frac{1}{3}}(\sqrt{a^2+b^2}-\sqrt{2}a)C_1^{-\frac{1}{3}},\, 
    \epsilon^3, \,
    \left( \frac{\sqrt{86\pi}}{a}\right)^3, \,
    2\epsilon,\,
    \left[ \frac{344\pi}{(1-k)b^2}\right]^{\frac{3}{2}},\, \\
    &\frac{8\sqrt{2}(b^3-a^3)C_2}{3(\sqrt{a^2+b^2}-\sqrt{2}a)},\,
    \frac{2\sqrt{2}b^2C_2}{\pi (\sqrt{a^2+b^2}-\sqrt{2}a)}
    \}.        
\end{aligned}
\end{equation}

\item For $(r, w, l) \in S(0)$, we want
\begin{equation}\label{RVP-initial-radius}
   a_0 - \epsilon < r < a_0 + \epsilon.
\end{equation}
Moreover, the initial charge density $\rho_0$ satisfies
\begin{equation}\label{RVP-initial-density-upper-bound}
    \rho_0(r) \le \frac{1}{a_0} \quad \text{for all } r >0,
\end{equation}
and 
\begin{equation} \label{RVP-initial-density-lower-bound}
    \rho_0(r) = \frac{1}{2a_0} \quad \text{for } r \in [a_0 - \frac{1}{2}\epsilon, \, a_0 + \frac{1}{2}\epsilon].
\end{equation}

\item Denote 
\[
N = \frac{a_0 -\epsilon - a}{a_0 +\epsilon- \sqrt{k}b}.
\]
Our definition of $\epsilon$ and choice of $a_0$ informs us that $0 \le a_0 + \epsilon -\sqrt{k}b < a_0 - \epsilon - a$. Therefore $N > 1$.\\
We then pick constant $h>0$ such that 
\[
h \le \{\frac{N-1}{2(N+1)}a_0^2,\, \frac{\sqrt{N}-1}{1+\sqrt{N}}a_0^2 \}
\]
We then demand that for $(r, w, l)$, 
\begin{equation} \label{RVP-initial-velocity}
-a_0^3 - ha_0 < w < -a_0^3 + ha_0.
\end{equation}

\item For $(r, w, l) \in S(0)$, we want 
\begin{equation} \label{RVP-initial-angular-momentum}
0\le l < (ha_0-\epsilon)(a_0-\epsilon)^2
\end{equation}

\item Finally, we choose time $T>0$ such that 
\begin{equation} \label{RVP-time}
(a_0+ \epsilon-\sqrt{k}b) \frac{\sqrt{2+(a_0^3 + ha_0)}}{a_0^3 - ha_0} \le T \le (a_0-\epsilon-a) \frac{\sqrt{1+(a_0^3- ha_0)}}{a_0^3 +ha_0}.
\end{equation}

\end{enumerate}

\begin{remark}
    We first verify that the inequality in (v) is in fact valid. To do so, we will need to show that 
    \[
    (a_0+ \epsilon-\sqrt{k}b) \frac{\sqrt{2+(a_0^3 + ha_0)}}{a_0^3 - ha_0}  \le (a_0-\epsilon-a) \frac{\sqrt{1+(a_0^3- ha_0)}}{a_0^3 +ha_0},
    \]
    that is,
    \[
    \frac{a_0^3 +ha_0}{a_0^3-ha_0} \frac{\sqrt{2+(a_0^3+ha_0^3)^2}}{\sqrt{1+(a_0^3-ha_0)^2}} \le \frac{a_0-\epsilon-a}{a_0+\epsilon-\sqrt{k}b} = N.
    \]
    This is where our choice of $h$ in (iii) comes into play. We can see that since $h < \frac{\sqrt{N}-1}{1+\sqrt{N}}a_0^2$, we have $(1+\sqrt{N})ha_0 < (\sqrt{N}-1)a_0^3$. Thus $2ha_0 < (\sqrt{N}-1)(a_0^3-ha_0)$, so
    \[
    1+ \frac{ha_0}{a_0^3-ha_0}= \frac{a_0^3+ha_0}{a_0^3-ha_0} \le \sqrt{N}.
    \]
    Moreover, since $h < \frac{N-1}{2(N+1)}a_0^2$, we have $2ha_0^4(N+1) \le (N-1)a_0^6$. Therefore $0 \le (N-1)a_0^6 - 2ha_0^4(N+1)$. So
    \[
    0\le (N-1) + (N-1)a_0^6 - 2ha_0^4(N+1) + (N-1)h^2a_0.
    \]
    Therefore
    \[
    1+(a_0^3 + ha_0)^2 \le 2N + N (a_0^3-ha_0)^2,
    \]
    so
    \[
    \frac{2+(a_0^3+ha_0^3)^2}{1+(a_0^3-ha_0)^2} \le N.
    \]
    Thus 
    \[
    \frac{a_0^3 +ha_0}{a_0^3-ha_0} \frac{\sqrt{2+(a_0^3+ha_0^3)^2}}{\sqrt{1+(a_0^3-ha_0)^2}} \le \sqrt{N}\times \sqrt{N} = \frac{a_0-\epsilon-a}{a_0+\epsilon-\sqrt{k}b},
    \]
    and the claim we make in (v) about the choice of $T$ is valid. 
\end{remark}

\begin{remark}
    To see that these choices of parameters are feasible, we now, similar to the VP case, provide an example of a smooth function that satisfy all the above conditions.\\
    \\
    First of all, based on the given constants $C_1$, $C_2$, $a$, $b$, and $c$, we first fix constants $k$, $\epsilon$, and $a_0$ that satisfies the requirements in (i), and choose constant $T>0$ that satisfies (iii).\\
    Next, we choose $H: \, [0, \infty) \rightarrow [0, \infty)$ to be a smooth function that satisfies 
    \[
    \int_{\mathbf{R}^3} H(|u|^2) du = \frac{1}{2a_0}
    \]
    with $supp(H) \subset [0, 1]$. We then re-scale this function with a scalar $\delta > 0$ by defining
    \[
    H_\delta (|u|^2) = \frac{1}{\delta^3} H(\frac{|u|^2}{\delta^2}).
    \]
    Then we will find that 
    \[
    \int_{\mathbf{R}^3} H_\delta(|u|^2) du = \frac{1}{2a_0}
    \]
    and $supp(H_\delta) \subset [0, \delta^2]$. \\
    \\
    Next, we choose a smooth cut-off function $\chi \in C^\infty ((0, \infty), [0, 1])$ such that 
    \[  \begin{cases} 
      \chi(r) = 1 &\text{for   } \quad a_0 -\frac{1}{2}\epsilon \le |r| \le a_0 + \frac{1}{2}\epsilon \\
      \chi(r) = 0 & \text{for   } \quad|r| \le a_0 - \epsilon \text{ or } |r| \ge a_0 + \epsilon \\
   \end{cases}
    \]
    We then fix constants $d = a_0^3-a_0$, and $\delta = ha_0 - \epsilon >0$, and define function
    \[
    f_0(x, v) = H_\delta (|\left(1+\frac{d}{|x|} \right)x + v|^2) \chi(|x|).
    \] 
    Clearly $f_0$ is smooth. We then move on to claim that it satisfies (ii), (iii), and (iv)\\
    \\
     To verify this, we first notice that since $r = |x|$, we can obtain from the definition of $\chi$ that 
    \[
    a_0 - \epsilon \le r \le a_0 + \epsilon.
    \]
    Next, the support of $h_\delta$ indicates us that
    \[
    |\left(1+\frac{d}{|x|} \right)x+v|^2 \le \delta^2.
    \]
    Using change of variables and the relation $|v|^2 = w^2 + lr^{-2}$, we can obtain that 
    \[
    (r +d + w)^2 +lr^{-2} \le \delta^2.
    \]
    This inequality indicates two things. Firstly, we have $|r+d+w| \le \delta$, that is, 
    \[
    -a_0^3-ha_0 \le -(a_0^3-a_0)-(a_0+\epsilon)-(ha_0-\epsilon) \le-d-r-\delta \le w,
    \]
    and 
    \[
    w\ge -d-r+\delta \ge -(a_0^3-a_0)-(a_0-\epsilon)+ (ha_0-\epsilon) \ge -a_0^3+ha_0.
    \]
    Therefore the conditions in (iii) are satisfied.\\
    Secondly, it also informs us that $lr^{-2}\le \delta $. Therefore
    \[
    l \le (ha_0-\epsilon) (a_0-\epsilon)^2
    \]
    and (iv) is verified.\\
    \\Finally, to check the solution satisfies (ii), we notice that 
    \[
    \int_{\mathbf{R}^3} H_\delta (|\left(1+\frac{d}{|x|} \right)x + v|^2) dv = \int_{\mathbf{R}^3} H_\delta (|u|^2)du = \frac{1}{2a_0}.
    \]   
    and thus we can calculate that 
    \[
    \rho_0 (x) = \int_{\mathbf{R}^3} f_0 (x, v) = \left( \int_{\mathbf{R}^3} H_\delta (|\left(1+\frac{d}{|x|} \right)x + v|^2) dv \right) \chi(|x|)  = \frac{1}{2a_0} \chi(|x|),
    \]
    which satisfies (ii) as we want.\\
    Therefore, we've  verified $f_0(x,v)$ to be a smooth function that satisfies all the settings for the initial data. 
    \end{remark}
    
\begin{remark}
    Since our choices of initial spatial radius $r$ and initial charge density $\rho_0$ in (ii) are identical to those in the VP case, we can derive the same estimation on the total particle mass $M$ using $a_0$ and $\epsilon$. Specifically, we have
\begin{equation}\label{RVP-total-mass}
    2\pi a_0 \epsilon + \frac{\pi}{6} a_0^{-1} \epsilon^3 \le M \le 8\pi a_0 \epsilon + \frac{8\pi}{3}a_0^{-1} \epsilon^3.
\end{equation}
\end{remark}

\subsection*{4.2. Proof of Theorem 7}
\begin{proof}
The proof of Theorem 7 shares a similar structure with that of Theorem 6, i.e. the VP case. We start with examining the spatial radius, charge density, and electric field at time zero. Next we show that our choice of $T$ enable the concentration to happen while the particles are moving inwards. After that, we use the estimations of $\mathscr{R}(T)$ in Lemma 4 to show that the particles are concentrated within the targeting shield. Finally we use Lemma 5 to verify the conditions of charge density and electric field at time $T$. \\
\\
To begin with, we first verify that the particle distribution, charge density, and electric field at time zero satisfies the requirements in Theorem 7. For $(r, w, l) \in S(0)$, we know
    \[
    r > a_0 - \epsilon = a_0 - \frac{\sqrt{a^2+b^2}- \sqrt{2}a}{4\sqrt{2}} >c
    \]
    by our choice of $a_0$.\\
Moreover, since $a_0 \ge C_1^{-1}$, our setting of the initial charge density in \ref{RVP-initial-density-lower-bound} informs us that 
\[
\lVert \rho (0) \rVert_{\infty} \le \frac{1}{a_0} \le C_1.
\]
Next we check that the electric field is small enough at time zero. Since our choice of the initial spatial radius and initial charge density in \ref{RVP-initial-radius}, \ref{RVP-initial-density-upper-bound}, and \ref{RVP-initial-density-lower-bound} is exactly the same as that in the VP case \ref{initial-radius}, \ref{initial-density-upper-bound}, and \ref{initial-density-lower-bound}, we can simply apply an identical argument about the electric field in each region at time zero, namely that 
\[
\begin{cases}
    |E(0, x)| = 0 \quad \quad \text{for } |x| \le a_0 - \epsilon,\\
    |E(0, x)| \le 32 \pi a_0^{-1} \epsilon + \frac{32\pi }{3}a_0^{-3}\epsilon^3 \quad \quad \text{for } a_0 -\epsilon \le |x| \le a_0 + \epsilon,\\
    |E(0, x)| \le 8 \pi a_0^{-1} \epsilon + \frac{8\pi }{3}a_0^{-3}\epsilon^3 \quad \quad \text{for } |x| \ge a_0 + \epsilon.
\end{cases}
\]
Again, an uniform bound then becomes
\[
\lVert E(0) \rVert_{\infty} \le 32\pi a_0^{-1}\epsilon + \frac{32\pi}{3}a_0^{-3}\epsilon^3.
\]
Then by choosing $a_0$ large enough we ensures that this value is smaller than $C_1$. Since we choose $a_0 \ge \text{max}\{ 6\sqrt{2}\pi(\sqrt{a^2+b^2}-\sqrt{2}a)C_1^{-1},  \, 2^{\frac{5}{6}}\pi^{1\over 3}(\sqrt{a^2+b^2}-\sqrt{2}a)C_1^{-\frac{1}{3}} $ in \ref{a_0-large-enough} and set $ \epsilon = \frac{(\sqrt{a^2+b^2}-\sqrt{a})}{4\sqrt{2}}$ in (i), we can estimate the two terms in the right side separately that  
\[
32\pi a_0^{-1}\epsilon = 4\sqrt{2}\pi(\sqrt{a^2+b^2}-\sqrt{2}a)a_0^{-1} \le \frac{2}{3}C_1
\]
and 
\[
\frac{32\pi}{3}a_0^{-3}\epsilon^3 = \frac{\sqrt{2}\pi}{24}(\sqrt{a^2+b^2}-\sqrt{2}a)^3a_0^{-3}\le \frac{1}{3}C_1.
\] 
Therefore, we can deduce that
\[
\lVert E(0) \rVert_{\infty} \le 32\pi a_0^{-1}\epsilon + \frac{32\pi}{3}a_0^{-3}\epsilon^3 \le C_1. 
\]
Thus we’ve proven that our solution satisfies the conditions on the initial particle spatial radius, charge density, and electric field.\\
\\
We now look at the particle distribution at time $T$. Similar to the case of the VP system, we want to verify that the concentration happens during the period which the particles move inwards, namely $T < T_0$. Now Lemma 4 informs us that 
\[
T_0 \ge r \left( 1- \frac{D}{r^2w^2 + D}\right) \ge r - \frac{\sqrt{D}}{|w|} ,
\]
where 
\[
D = l + Mr \sqrt{1 + w^2 + lr^{-2}}.
\]
Therefore our first step is to derive a reasonable estimation for $D$. Our choice of initial data in \ref{RVP-initial-velocity} and \ref{RVP-initial-angular-momentum} tells us that 
\begin{equation}
    \begin{aligned}
        D & = l + Mr \sqrt{1+w^2 + lr^{-2}}\\
        & \le (ha_0-\epsilon)(a_0-\epsilon)^2 + (8\pi a_0 \epsilon + \frac{8\pi }{3}a_0^{-1}\epsilon^3) (a_0 + \epsilon) \sqrt{2 + (a_0^3+ha_0)^2}\\
        &\le (ha_0-\epsilon)(a_0-\epsilon)^2 + (8\pi a_0^2 \epsilon + \frac{8\pi}{3}\epsilon^3 + 8\pi a_0 \epsilon^2 + \frac{8\pi}{3}a_0^{-1}\epsilon^4) \sqrt{2 + (a_0^3+ha_0)^2}
    \end{aligned}
\end{equation}
by \ref{RVP-initial-radius} and \ref{RVP-total-mass}. Moreover, since we've set $a_0 \ge 1$ and $a_0\ge \epsilon^3$ (that is, $\epsilon \le a_0^{\frac{1}{3}}$), we can simplify the inequality that 
\begin{equation}
    \begin{aligned}
        D &\le (ha_0)^2 + (8\pi a_0^{\frac{7}{3}} + \frac{8\pi}{3}a_0 + 8\pi a_0^{\frac{5}{3}} + \frac{8\pi}{3}a_0^{\frac{1}{3}}) \sqrt{2+ (a_0^3+ha_0)^2}\\
        &\le (2a_0^2 )^2 + (8\pi a_0^{\frac{7}{3}} + \frac{8\pi}{3}a_0^{\frac{7}{3}} + 8\pi a_0^{\frac{7}{3}} + \frac{8\pi}{3}a_0^{\frac{7}{3}}) \sqrt{2+ (a_0^3+2a_0^2)^2}\\
        &\le 4a_0^4 + \frac{64\pi}{3} a_0^\frac{7}{3} \sqrt{2 + 9a_0^6}\\
        &\le 4a_0^4 + \frac{64\pi}{3} a_0^\frac{7}{3} \sqrt{16a_0^6}\\
        &\le 4a_0^2 + \frac{256\pi}{3}a_0^{\frac{16}{3}}\\
        &\le 86\pi a_0^{\frac{16}{3}}.
    \end{aligned}
\end{equation}
With this estimation in hand, we can now verify our choice of focusing time $T$. We first see that if we have $\frac{\sqrt{1+(a_0^3-ha_0)^2}}{a_0^3+ha_0} \le 1$, then \[
T \le (a_0-\epsilon-a) \frac{\sqrt{1+(a_0^3-ha_0)^2}}{a_0^3+ha_0} \le a_0-\epsilon-a.
\]
Then since we've set $a_0 \ge \left( \frac{\sqrt{86\pi}}{a}\right)^3$, that is, $\sqrt{86\pi}a_0^{-\frac{1}{3}} \le a$, we have
\[
\frac{\sqrt{D}}{|w|} \le \frac{\sqrt{86\pi a_0^{\frac{16}{3}}}}{a_0^3} \le  \frac{\sqrt{86\pi}a^{\frac{8}{3}}}{a_0^3} \le \sqrt{86\pi}a_0^{-\frac{1}{3}} \le a. 
\]
Therefore we can deduce that 
\[
T \le a_0 - \epsilon - a \le (a_0 - \epsilon) - \frac{\sqrt{D}}{|w|} \le r - \frac{\sqrt{D}}{|w|} \le T_0.
\]
To make this happen, we will need $1+(a_0^3-ha_0)^2 \le (a_0^3+ha_0)$, that is, $4ha_0^4 \ge 1$. Since we've set $a_0$ large enough, this condition can be satisfied. Therefore, $T \le T_0$ is verified, and the concentration in fact happens while the particles are moving inwards. \\
\\
We now turn to the location of particles at time $T$. For all $(r, w, l) \in S(0)$, we want $a \le \mathscr{R}(T, 0, r, w, l) \le b$, so that the particles are distributed in the targeting range. To do so, we will need the estimations on $\mathscr{R}(T)$ we derived in Lemma 4. Our choice of parameters in \ref{RVP-initial-radius}, \ref{RVP-initial-velocity}, and \ref{RVP-initial-angular-momentum} informs us that 
\[
\frac{|w|}{\sqrt{1+w^2+ lr^{-2}}}T \le \frac{a_0^3+ha_0}{\sqrt{1+(a_0^3-ha_0)^2}} \times (a_0-\epsilon-a) \frac{\sqrt{1+(a_0^3-ha_0)^2}}{a_0^3+ha_0} \le a_0 - \epsilon -a.
\]
Therefore by Lemma 4, 
\begin{equation}
    \begin{aligned}
        \mathscr{R}(T) & \ge \frac{l}{2r^3 (1 + w^2 + lr^{-2})} T^2 + \frac{w}{\sqrt{1+w^2 + lr^{-2}}}T + r\\
        &\ge  r - \frac{|w|}{\sqrt{1+w^2 + lr^{-2}}}T\\
        &\ge (a_0 - \epsilon) - (a_0 - \epsilon -a) \\
        &\ge a. 
    \end{aligned}
\end{equation}
Similarly, we know 
\[
\frac{|w|}{\sqrt{1+w^2 + lr^{-2}}}T \ge \frac{a_0^3-ha_0}{\sqrt{2+(a_0^3+ha_0)}} \times (a_0+ \epsilon-\sqrt{k}b) \frac{\sqrt{2+(a_0^3 + ha_0)}}{a_0^3 - ha_0} \ge a_0 + \epsilon - \sqrt{k}b.
\]
Therefore 
\begin{equation}
    \begin{aligned}
        \left(r - \frac{|w|}{\sqrt{1+ w^2 + lr^{-2}}} T \right)^2 & \le \left(a_0 + \epsilon - (a_0 + \epsilon - \sqrt{k}b) \right)^2\\
        & \le kb^2. 
    \end{aligned}
\end{equation}
Moreover, since we know $T \le a_0$, and we've set $a_0 \ge 2 \epsilon$ (so that $a_0 - \epsilon \ge \frac{1}{2}a_0)$, we can also calculate 
\begin{equation}
    \begin{aligned}
        \frac{D}{r^2 (1+w^2 + lr^{-2})}T^2 &\le \frac{D}{(a_0 -\epsilon)^2 (1 + (a_0^3+ha_0)^2)}a_0^2\\
        &\le \frac{D}{\frac{1}{4}a_0^2 \times a_0^6}a_0\\
        &\le \frac{86\pi a_0^{\frac{16}{3}}}{\frac{1}{4}a_0^8} a_0^2\\
        &\le 344\pi a_0^{-\frac{2}{3}}.
    \end{aligned}
\end{equation}
We've chosen $a_0$ large enough such that $a_0 \ge \left[ \frac{344\pi}{(1-k)b^2}\right]^{\frac{3}{2}}$, which is equivalent to say $344\pi a_0^{-\frac{2}{3}} \le (1-k)b^2$. Then we can deduce that 
\[
\frac{D}{r^2 (1+w^2+lr^{-2}) }T^2 \le (1-k) b^2.
\]
Finally using Lemma 4, we can reach the conclusion that since $T \in [0, T_0)$, 
\begin{equation}
    \begin{aligned}
        \mathscr{R}(T) & \le \left(r- \frac{|w|}{\sqrt{1+w^2+lr^{-2}}} T\right)^2 + \frac{D}{r^2 (1+w^2+lr^{-2})} T^2\\
        &\le kb^2 + (1-k)b^2\\
        &\le b^2,
    \end{aligned}
\end{equation}
so that $\mathscr{R}(T) \le b$. Therefore, for $(r, w, l) \in S(0)$, we've verified that 
\[
a \le \mathscr{R}(T, 0, r, w, l) \le b.
\]
\hfill\break
Finally, we're left to verify that the charge density and electric field are large enough at time $T$. Here again we apply an identical argument as in the VP case. We use Lemma 5 to deduce that the conditions $\lVert \rho(T) \rVert_\infty$, $\lVert E(T) \rVert_\infty \ge C_2$ are satisfied if we have 
\[
M \ge max\{\frac{4\pi}{3}(b^3-a^3)C_2,\, b^2 C_2\}.
\]
We know that $M \ge 2\pi a_0\epsilon + \frac{\pi}{6}a_0^{-1}\epsilon^3$. Then since we've set $a_0$ large enough, namely 
\[
a_0 \ge max\{\frac{8\sqrt{2}(b^3-a^3)C_2}{3(\sqrt{a^2+b^2}-\sqrt{2}a)}, \,  \frac{2\sqrt{2}b^2C_2}{\pi (\sqrt{a^2+b^2}-\sqrt{2}a)}\} 
\]
we can easily verify that 
\[
M \ge 2\pi a_0\epsilon \ge \frac{(\pi \sqrt{a^2+b^2}-\sqrt{2}a)}{2\sqrt{2}}a_0 \ge max \{\frac{4\pi}{3}(b^3-a^3), \, b^2C_2\}.
\]
Thus by Lemma 5, $\lVert \rho(T) \rVert_\infty \ge C_2$ and $\lVert E(T) \rVert_\infty \ge C_2$ is verified. Therefore, Theorem 7 is proven. 
\end{proof}


\newpage
\begin{thebibliography}{unsrt}
\bibitem{Arbitrarily large solutions of the Vlasov-Poisson system}
    J. Ben-Artzi, S. Calogero and S. Pankavich, \emph{Arbitrarily large solutions of the Vlasov-Poisson system}, SIAM Journal on Mathematical Analysis (2018) 50(4): 43114326
\bibitem{Concentrating solutions of the relativistic Vlasov- Maxwell system}
    J. Ben-Artzi, S. Calogero and S. Pankavich, \emph{Concentrating solutions of the relativistic Vlasov- Maxwell system}, Commun. Math. Sci. (to appear), preprint - arxiv: 1807.02801
\bibitem{Symmetric plasmas and their decay}
E. Horst, \emph{Symmetric plasmas and their decay}, Comm. Math. Phys. (1990) 126:613-633
\bibitem{Focusing solutions of the Vlasov-Poisson System}
Z. Zhang, \emph{Focusing solutions of the Vlasov-Poisson System}, math.AP. (2019) arxiv: 1901.10639
\bibitem{The Cauchy problem in kinetic theory}
Robert T.Glassey, \emph{The Cauchy problem in kinetic theory}.  (1996). Society for Industrial and Applied Mathematics.
\bibitem{Propogation of moments and regularity for the three dimensional Vlasov-Poisson system}
P. L. Lions and B. Perthame, \emph{Propogation of moments and regularity for the three dimensional Vlasov-Poisson system}, Invent. Math. (1991) 105:415-430.
\bibitem{Global classical solution of the Vlasov-Poisson system in three dimensions for general initial data}
K. Pfaffelmoser, \emph{Global classical solution of the Vlasov-Poisson system in three dimensions for general initial data}, J. Diff. Eq. (1992) 95(2):281-303.
\end{thebibliography}
\end{document}